\documentclass{amsart}

\usepackage{amssymb}
\usepackage{amsmath}
\usepackage{cite}

\newtheorem{theo}{Theorem}
\newtheorem{theorem}{Theorem}[section]
\newtheorem*{theorem*}{Theorem}
\newtheorem{lemma}[theorem]{Lemma}
\newtheorem*{coro*}{Corollary}

\title{\phantom{-} \ General $L_p$ affine isoperimetric inequalities}

\author{Christoph Haberl and Franz E.\ Schuster}

\begin{document}

\renewcommand{\citeleft}{{\rm [}}
\renewcommand{\citeright}{{\rm ]}}
\renewcommand{\citepunct}{{\rm,\ }}
\renewcommand{\citemid}{{\rm,\ }}

%\begin{abstract}
%\hspace{-0.52cm} Sharp $L_p$ affine isoperimetric inequalities are
%established for the entire class of $L_p$ projection bodies and
%the entire class of $L_p$ centroid bodies. \linebreak These new
%inequalities strengthen the $L_p$ Petty projection and the
%\linebreak $L_p$ Busemann--Petty centroid inequality.
%\end{abstract}

\maketitle

\vspace{1cm}

\section{Introduction}
Projection bodies were introduced by Minkowski at the turn of the
previous century and have since become a central notion in convex
geometry. They arise naturally in a number of different areas
such as functional analysis, stochastic \linebreak geometry and
geometric tomography, see e.g.,
\textbf{\cite{bourgain:lindenstrauss, gardner2ed, grinberg:zhang,
Ludwig:projection, Schneider:Weil83, steineder08, Thompson}}.
\linebreak The fundamental affine isoperimetric inequality for
projection bodies is the {\it Petty projection inequality}
\textbf{\cite{petty67}}: Among all convex bodies of given volume,
the ones whose polar projection bodies have maximal volume are
precisely the ellipsoids. This inequality turned out to be far
stronger than the classical \linebreak isoperimetric inequality.
Lutwak, Yang, and Zhang \textbf{\cite{LYZ2000}} (see also Campi
and Gronchi \textbf{\cite{Campi:Gronchi02a}}) established an
important $L_p$ Petty projection inequality for the (symmetric)
$L_p$ analogue of the projection operator. This extension is the
geometric core of a sharp affine $L_p$ Sobolev inequality which
is significantly stronger than the classical $L_p$ Sobolev
inequality, see \textbf{\cite{LYZ2002, Zhang99}}. Recent advances
\linebreak in valuation theory by Ludwig
\textbf{\cite{Ludwig:Minkowski}} revealed that the $L_p$
projection operator used in \textbf{\cite{LYZ2000}} is only one
representative of an entire class of $L_p$ extensions of the
classical projection operator. In this article we establish the
$L_p$ Petty \linebreak projection inequality for each member of
the family of $L_p$ projection operators. \linebreak It is shown
that each of these new inequalities strengthens and implies the
\linebreak previously known $L_p$ Petty projection inequality.
Moreover, the two strongest inequalities are identified. Similar
results for the $L_p$ Busemann--Petty centroid inequality are
also established.

The celebrated {\em Blaschke--Santal\'{o} inequality} is by far
the best known affine isoperimetric inequality (see e.g.,
\textbf{\cite{gardner2ed, Gruber:CDG, Schneider:CB}}): The product
of the volumes of polar reciprocal convex bodies is maximized
precisely by ellipsoids. Lutwak and Zhang \textbf{\cite{LZ1997}}
obtained an important $L_p$ version of the Blaschke--Santal\'{o}
inequality. Their inequality includes as a limiting case the
classical inequality for origin-symmetric convex bodies. For
convex bodies which are not origin-symmetric this $L_p$ extension
yields an inequality which is weaker than the
Blaschke--Santal\'{o} inequality. As an application of our work,
we establish the correct $L_p$ analog of the
Blaschke--Santal\'{o} inequality, one that includes as a limiting
case the classical inequality for all convex bodies.

\pagebreak

For a convex body $K$ (i.e., a nonempty, compact convex subset of
$\mathbb{R}^n$) denote \linebreak by $h(K,x)=\max \{x \cdot y: y
\in K\}$, for $x \in \mathbb{R}^n$, the support function of $K$.
The {\it projection body} $\Pi K$ of $K$ is the convex body whose
support function in the \linebreak direction $u$ is equal to the
$(n-1)$-dimensional volume of the projection of $K$ onto the
hyperplane orthogonal to $u$. An important recent result by Ludwig
\textbf{\cite{Ludwig:Minkowski}} has demonstrated the special
place of projection bodies in the affine theory \linebreak of
convex bodies: The projection operator was characterized as the
unique Minkowski valuation which is contravariant with respect to
nondegenerate \linebreak linear transformations.

A function $\Phi$ defined on a subset $\mathcal{L}$ of the set of
convex bodies $\mathcal{K}^n$ and taking values in an abelian
semigroup is called a {\it valuation} if
\begin{equation} \label{defval}
\Phi(K \cup L) + \Phi(K \cap L) = \Phi K + \Phi L,
\end{equation}
whenever $K, L, K \cap L, K \cup L \in \mathcal{L}$. The theory
of real valued valuations lies \nolinebreak at the core of
geometry. They were the critical ingredient in Dehn's solution of
Hilbert's third problem. For information on the classical theory
of valuations, see \textbf{\cite{Klain:Rota}} and
\textbf{\cite{McMullen93}}. For some of the more recent results,
see \textbf{\cite{Alesker99, Alesker01, Alesker03, Bernig03,
Ludwig:projection, Ludwig:matrix, Ludwig:Minkowski, Ludwig06,
LR99, centro}}.

First results on convex body valued valuations were obtained by
S\-chneider \textbf{\cite{Schneider74}} in the 1970s, where the
addition of convex bodies in (\ref{defval}) is Minkowski addition
defined by $h(K+L,\cdot)=h(K,\cdot) + h(L,\cdot)$, see also
\textbf{\cite{Kiderlen, schneider:schuster:projection, Schu06a}}.
In recent years the investigations of these {\em Minkowski
valuations} gained momentum through a series of articles by
Ludwig \textbf{\cite{Ludwig:projection, Ludwig:Minkowski}}. She
obtained complete classifications of Minkowski valuations
compatible with nondegenerate linear transformations (see Section
3 for precise definitions).

Projection bodies are part of the classical Brunn--Minkowski
theory which is the result of joining the notion of volume with
the usual vector addition of convex sets. The books by Gardner
\textbf{\cite{gardner2ed}}, Gruber \textbf{\cite{Gruber:CDG}} and
Schneider \textbf{\cite{Schneider:CB}} form an excellent
introduction to the subject. In a series of articles
\textbf{\cite{Lutwak93b, Lutwak96}}, Lutwak showed that merging
the notion of volume with the $L_p$ Minkowski addition of convex
sets, introduced by Firey, leads to a Brunn--Minkowski
\nolinebreak theory for each $p \geq 1$. Since Lutwak's seminal
work, the topic has been the focus of intense study, see e.g.,
\textbf{\cite{Chou:Wang:06, Gardner99, Gardner:Zhang,
Ludwig:Minkowski, centro, Lutwak:Oliker, LYZ2000, LYZ2006,
LYZ2002, LYZ2004, LZ1997, Ryabogin:Zvavitch, schuett:werner:04,
Stancu02, Stancu03}}.

For $p > 1$, Ludwig \textbf{\cite{Ludwig:Minkowski}} introduced a
two-parameter family of convex bodies,
\begin{equation} \label{compipp}
c_1 \cdot \Pi_p^+K +_p c_2 \cdot \Pi_p^-K, \qquad K \in
\mathcal{K}^n_{\mathrm{o}},
\end{equation}
and established the $L_p$ analogue of her classification of the
projection operator: \linebreak She showed that the convex bodies
defined in ({\ref{compipp}) constitute all of the $L_p$
extensions of projection bodies. Here,
$\mathcal{K}^n_{\mathrm{o}}$ is the set of convex bodies which
contain the origin in their interiors and $c_1, c_2\geq0$ (not
both zero). The convex body \linebreak defined by (\ref{compipp})
is an $L_p$ Minkowski combination of the {\em nonsymmetric $L_p$
\linebreak projection bodies} $\Pi_p^{\pm}K$ (see Sections 2 and 3
for definitions).

The (symmetric) {\it $L_p$ projection body $\Pi_p K$} of $K \in
\mathcal{K}^n_{\mathrm{o}}$, first defined in
\textbf{\cite{LYZ2000}}, \nolinebreak is
\[\Pi_pK=\mbox{$\frac{1}{2}$} \cdot \Pi_p^+K+_p \mbox{$\frac{1}{2}$} \cdot  \Pi_p^-K.  \]
As our main result we extend the $L_p$ Petty projection
inequality for $\Pi_p$ by Lutwak, Yang, and Zhang to the entire
class (\ref{compipp}) of $L_p$ projection bodies.

Let $K^*=\{x \in \mathbb{R}^n: x \cdot y \leq 1 \mbox{ for all } y
\in K \}$ denote the polar body of $K \in
\mathcal{K}_{\mathrm{o}}^n$. We use $V(K)$ to denote the volume
of $K$ and we write $B$ for the Euclidean unit ball. If $\Phi:
\mathcal{K}_{\mathrm{o}}^n \rightarrow
\mathcal{K}_{\mathrm{o}}^n$, we use $\Phi^*K$ to denote $(\Phi
K)^*$.

\begin{theo}\label{main1} Let $K \in \mathcal{K}^n_{\mathrm{o}}$ and $p > 1$. If $\Phi_p K$ is the convex body defined \nolinebreak by
\[\Phi_p K =  c_1\cdot \Pi_p^+K+_pc_2\cdot \Pi_p^-K, \]
where $c_1, c_2 \geq 0$ are not both zero, then
\[V(K)^{n/p-1}V(\Phi_p^* K)\leq V(B)^{n/p-1}V(\Phi_p^* B),\] with
equality if and only if $K$ is an ellipsoid centered at the
origin.
\end{theo}

The case $\Phi_p = \Pi_p$ of Theorem \ref{main1} is the $L_p$ Petty
projection inequality by Lutwak, Yang, and Zhang.

The natural problem arises to determine for fixed $K \in
\mathcal{K}^n_{\mathrm{o}}$ the extreme values of $V(\Phi_p^* K)$
among all suitably normalized (e.g., satisfying $\Phi_p B =
B$)\linebreak $L_p$ projection bodies (\ref{compipp}). Here, we
will show that for $K \in \mathcal{K}^n_{\mathrm{o}}$,
\[
V(\Pi_p^* K) \leq V(\Phi_p^* K) \leq V(\Pi_p^{\pm,*}K).
\]
If $K$ is not origin-symmetric and $p$ is not an odd integer,
these inequalities are strict unless $\Phi_p = \Pi_p$, or $\Phi_p =
\Pi_p^{\pm}$, respectively. This shows that each of the new
inequalities established in Theorem \ref{main1} strengthens and
implies the previously known $L_p$ Petty projection inequality and
that the nonsymmetric operators $\Pi_p^{\pm}$ (and their
multiples) give rise to the strongest inequalities.

\vspace{0.3cm}

Centroid bodies (volume normalized moment bodies) are a classical notion from geometry which have
attracted increased attention in recent years, see e.g.,
\textbf{\cite{gardner2ed, grinberg:zhang, lutwak86, lutwak90, LYZ2000}}. The {\it moment body} $\mathrm{M}K$ of a
convex body $K$ is the convex body defined by
\[h(\mathrm{M}K,u)=\int_{K} |u \cdot x |\,dx, \qquad u \in S^{n-1}.  \]
If $K$ has nonempty interior, then $\Gamma K =
V(K)^{-1}\mathrm{M} K$ is the {\it centroid body} of $K$.

Petty established the Petty projection inequality as a
consequence of the {\it Busemann--Petty centroid inequality}
\textbf{\cite{petty61}}: Among all convex bodies of given volume, the ones
whose centroid bodies have minimal volume are precisely the
ellipsoids. Lutwak, Yang, and Zhang \textbf{\cite{LYZ2000}} (see also Campi
and Gronchi \textbf{\cite{Campi:Gronchi02a}}) established the $L_p$ version
of the Busemann--Petty centroid inequality: For $p
> 1$ and convex bodies $K$ containing the origin in their
interiors,
\begin{equation}\label{lyzinequ2}
V(K)^{n/p-1}V(\mathrm{M}_pK)\leq V(B)^{n/p},
\end{equation}
with equality if and only if $K$ is an ellipsoid centered at the
origin. Here, $\mathrm{M}_pK$ denotes the (symmetric) {\it $L_p$
moment body}, defined in \textbf{\cite{LZ1997}} by
\[\mathrm{M}_pK=\mbox{$\frac{1}{2}$} \cdot
\mathrm{M}_p^+K+_p\mbox{$\frac{1}{2}$} \cdot \mathrm{M}_p^-K,  \]
where $\mathrm{M}_p^{\pm}K$ are the {\it nonsymmetric $L_p$ moment
bodies} (see Section 3). Since their introduction $L_p$ moment
bodies have become the focus of intense study, see e.g.,
\textbf{\cite{Campi:Gronchi02a, fleury, grinberg:zhang, haberl08,
haberl:ludwig, Ludwig:Minkowski, LYZ2000,
Yaskin:Yaskina_centroids}} and the noted paper
\textbf{\cite{paouris}}.

Ludwig \textbf{\cite{Ludwig:Minkowski}} characterized moment bodies as the
unique (non-trivial) \linebreak homogeneous Minkowski valuations
which intertwine volume preserving linear transformations. For $p
> 1$, Ludwig \textbf{\cite{Ludwig:Minkowski}} introduced and characterized the two-parameter family
\begin{equation}\label{commpp}
\mbox{$c_1$} \cdot
\mathrm{M}_p^+K+_p\mbox{$c_2$} \cdot \mathrm{M}_p^-K, \qquad K\in\mathcal{K}^n_{\mathrm{o}},
\end{equation}
as all of the possible $L_p$ analogues of moment bodies.

Our $L_p$ Busemann--Petty centroid inequality for the entire class (\ref{commpp})
of $L_p$ moment bodies is:

\begin{theo}\label{main2} Let $K \in \mathcal{K}^n_{\mathrm{o}}$ and $p > 1$. If $\Psi_p K$ is the convex body defined \nolinebreak by
\[\Psi_p K = c_1\cdot\mathrm{M}_p^+K+_pc_2\cdot\mathrm{M}_p^-K, \]
where $c_1, c_2 \geq 0$ are not both zero, then
\[V(K)^{-n/p-1}V(\Psi_p K)\geq V(B)^{-n/p-1}V(\Psi_p B),\] with
equality if and only if $K$ is an ellipsoid centered at the
origin.
\end{theo}

In fact, in Section 6 a stronger version of Theorem \ref{main2},
valid for all star bodies, will be established.

For $K\in\mathcal{K}^n_{\mathrm{o}}$ and suitably normalized
(e.g., satisfying $\Psi_p B=B$) $L_p$ moment bodies
(\ref{commpp}), we will show that
\[V(\mathrm{M}_p K) \geq V(\Psi_p K) \geq V(\mathrm{M}_p^{\pm}K). \]
If $K$ is not origin-symmetric and $p$ is not an odd integer,
these inequalities are strict unless $\Psi_p = \mathrm{M}_p$, or
$\Psi_p = \mathrm{M}_p^{\pm}$, respectively. Consequently, each of
the new inequalities established in Theorem \ref{main2}
strengthens and implies inequality (\ref{lyzinequ2}). Moreover,
the nonsymmetric operators $\mathrm{M}_p^{\pm}$ provide the
strongest version of the $L_p$ Busemann--Petty centroid
inequality.

Recall that for $K \in \mathcal{K}_{\mathrm{o}}^n$ the
Blaschke--Santal\'{o} inequality states
\[V(K)V(K^s) \leq V(B)^2,  \]
with equality if and only if $K$ is an ellipsoid. Here,
$K^s=(K-s)^*$ is the polar body of $K$ with respect to the
Santal\'{o} point $s$ of $K$, i.e., the unique point $s \in
\mathrm{int}\,K$ which minimizes $V((K-x)^*)$ among all
translates $K - x$, for $x \in \mathrm{int}\, K$. From Theorem
\ref{main2}, we obtain:

\begin{coro*} \label{bscor} If $\Psi_p$ is defined as in Theorem \ref{main2}, then for $K \in
\mathcal{K}^n_{\mathrm{o}}$,
\[V(K)^{n/p+1}V(\Psi_p^s K)\leq V(B)^{n/p+1}V(\Psi_p^s B),\]
with equality if and only if $K$ is an ellipsoid centered at the
origin.
\end{coro*}

Here, the case $\Psi_p = \mathrm{M}_p$ was established by Lutwak
and Zhang \nolinebreak \textbf{\cite{LZ1997}}. We remark that
$\mathrm{M}_p^+K$ converges to $K$ as $p \rightarrow \infty$.
Thus, as a limiting case we obtain for $\Psi_p=\mathrm{M}_p^+$ the
classical Blaschke--Santal\'{o} inequality.

\section{Background Material}

In the following we state the necessary background material. For
quick reference, we collect basic properties of $L_p$ mixed and
dual mixed volumes.

The setting for this article is Euclidean $n$-space
$\mathbb{R}^n$ with $n \geq 3$. We will also assume throughout
that $1 < p < \infty$. Thus, in the following we will omit these
restrictions on $n$ and $p$.

Associated with a convex body $K \in \mathcal{K}^n_{\mathrm{o}}$
is its surface area measure, $S(K,\cdot)$, on $S^{n-1}$. For a
Borel set $\omega \subseteq S^{n-1}$, $S(K,\omega)$ is the
$(n-1)$-dimensional Hausdorff measure of the set of all boundary
points of $K$ for which there exists a normal vector of $K$
belonging to $\omega$. By Minkowski's uniqueness theorem (see
e.g., \textbf{\cite[\textnormal{p.\ 397}]{Schneider:CB}}), the
convex body $K$ is determined up to translation by the measure
$S(K,\cdot)$.

We call a convex body $K \in \mathcal{K}^n_{\mathrm{o}}$ {\em
smooth} if its boundary is $C^2$ with everywhere positive
curvature. For a smooth convex body $K$, the surface area measure
$S(K,\cdot)$ is absolutely continuous with respect to spherical
Lebesgue measure:
\[dS(K,u)=f(K,u)\,du, \qquad u \in S^{n-1}.  \]
The positive continuous function $f(K,\cdot)$ is called the {\em
curvature function} of $K$. It is the reciprocal of the Gauss
curvature as a function of the outer normals.

For $p \geq 1$, $K, L
\in \mathcal{K}^n_{\mathrm{o}}$ and $\alpha, \beta \geq 0$ (not
both zero), the {\it $L_p$ Minkowski combination} $\alpha \cdot K
+_p \beta \cdot L$ is the convex body defined by
\[h(\alpha \cdot K +_p \beta \cdot L,\cdot)^p=\alpha h(K,\cdot)^p+\beta h(L,\cdot)^p.   \]
Introduced by Firey in the 1960's, this notion is the basis of
what has become known as the $L_p$ Brunn--Minkowski theory (or the
Brunn--Minkowski--Firey theory). Obviously, $L_p$ Minkowski and
the usual scalar multiplications are related by $\alpha\cdot
K=\alpha^{1/p}K$.

For $K, L \in \mathcal{K}^n_{\mathrm{o}}$, the $L_p$ mixed volume,
$V_p(K,L)$, was defined in \textbf{\cite{Lutwak93b}} by
\[\frac{n}{p}V_p(K,L)=\lim \limits_{\varepsilon \rightarrow 0^+} \frac{V(K +_p \varepsilon \cdot L)-V(K)}{\varepsilon}.  \]
Clearly, the diagonal form of $V_p$ reduces to ordinary volume,
i.e., for $K \in \mathcal{K}^n_{\mathrm{o}}$,
\begin{equation} \label{vol1}
V_p(K,K)=V(K).
\end{equation}
It was shown in \textbf{\cite{Lutwak93b}} that corresponding to each convex
body $K \in \mathcal{K}^n_{\mathrm{o}}$, there is a positive
Borel measure on $S^{n-1}$, the $L_p$ {\em surface area measure}
$S_p(K,\cdot)$ of $K$, such that for every $L \in
\mathcal{K}^n_{\mathrm{o}}$,
\begin{equation} \label{defvp}
V_p(K,L)=\frac{1}{n} \int_{S^{n-1}} h(L,u)^p dS_p(K,u).
\end{equation}
The measure $S_1(K,\cdot)$ is just the surface area measure of
$K$. Moreover, the $L_p$ surface area measure is absolutely
continuous with respect to $S(K,\cdot)$:
\begin{equation}\label{absc}
dS_p(K,u)=h(K,u)^{1-p}\,dS(K,u), \qquad u \in S^{n-1}.
\end{equation}
It was shown in \textbf{\cite{Lutwak93b}} that, if $K, L \in
\mathcal{K}^n_{\mathrm{o}}$ and $p \neq n$, then
\begin{equation*}
\begin{array}{lcl} S_p(K,\cdot)=S_p(L,\cdot) & \Longrightarrow &
K=L\phantom{\lambda \quad \, \lambda > 0.}
\end{array}
\end{equation*}
and, if $p=n$, then
\begin{equation*}
\begin{array}{lcl} S_n(K,\cdot)=S_n(L,\cdot) & \Longrightarrow & K=\lambda
L, \quad \lambda > 0.\end{array}
\end{equation*}
These uniqueness properties of the $L_p$ surface area measure are
consequences of the $L_p$ Minkowski inequality \textbf{\cite{Lutwak93b}}:
If $K, L \in \mathcal{K}^n_{\mathrm{o}}$, then
\begin{equation}\label{minkin}
V_p(K,L)^n \geq V(K)^{n-p}V(L)^{p},
\end{equation}
with equality if and only if $K$ and $L$ are dilates.

Firey's $L_p$ Brunn--Minkowski inequality states: If $K, L \in
\mathcal{K}^n_{\mathrm{o}}$, then
\begin{equation} \label{bmin}
V(K +_p L)^{p/n} \geq V(K)^{p/n} + V(L)^{p/n},
\end{equation}
with equality if and only if $K$ and $L$ are dilates.

For a compact set $L$ in $\mathbb{R}^n$ which is star-shaped with
respect to the origin, we denote by $\rho(L,x)=\max\{\lambda \geq
0: \,\lambda x\in L\}$, $x \in \mathbb{R}^n \backslash \{0\}$, the
radial function of $L$. If $\rho(L,\cdot)$ is positive and
continuous, we call $L$ a star body. The set of star bodies is
denoted by $\mathcal{S}^n$.

If $K \in \mathcal{K}^n_{\mathrm{o}}$ is a convex body, then it
follows from the definitions of support functions and radial
functions, and the definition of the polar body of $K$, that
\begin{equation} \label{suprad}
\rho(K^*,\cdot)=h(K,\cdot)^{-1} \qquad \mbox{and} \qquad
h(K^*,\cdot)=\rho(K,\cdot)^{-1}.
\end{equation}
For $\alpha, \beta \geq 0$ (not both zero), the {\it $L_p$
harmonic radial combination} $\alpha \cdot K \,\widetilde{+}_p\,
\beta \cdot L$ of $K, L \in \mathcal{S}^n$ is the star body
defined by
\[\rho(\alpha \cdot K\, \widetilde{+}_p \, \beta \cdot L,\cdot)^{-p}=\alpha \rho(K,\cdot)^{-p}+\beta \rho(L,\cdot)^{-p}.   \]
Although our notation does not reflect the obvious difference
between $L_p$ and dual $L_p$ scalar multiplication, there should
be no possibility of confusion. Clearly, $L_p$ harmonic radial and the usual scalar multiplications
are related by $\alpha\cdot K=\alpha^{-1/p}K$.

For convex bodies, Firey started investigations of harmonic $L_p$
combinations which were continued by Lutwak leading to a dual
$L_p$ Brunn--Minkowski theory. The dual $L_p$ mixed volume
$\widetilde{V}_{-p}(K,L)$ of $K, L \in \mathcal{S}^n$ was defined
in \textbf{\cite{Lutwak96}} by
\[-\frac{n}{p}\widetilde{V}_{-p}(K,L)=\lim \limits_{\varepsilon \rightarrow 0^+} \frac{V(K\, \widetilde{+}_p\, \varepsilon \cdot L)-V(K)}{\varepsilon}.  \]
Clearly, the diagonal form of $\widetilde{V}_{-p}$ reduces to
ordinary volume, i.e., for $L \in \mathcal{S}^n$,
\begin{equation}\label{vol2}
\widetilde{V}_{-p}(L,L)=V(L).
\end{equation}
The polar coordinate formula for volume leads to the following
integral representation of the dual $L_p$ mixed volume
$\widetilde{V}_{-p}(K,L)$ of the star bodies $K, L$:
\begin{equation} \label{defvminp}
\widetilde{V}_{-p}(K,L)=\frac{1}{n}\int_{S^{n-1}}\rho(K,u)^{n+p}\rho(L,u)^{-p}\,du.
\end{equation}
Here, integration is with respect to spherical Lebesgue measure.
An application of H\"older's integral inequality to
(\ref{defvminp}) yields the dual $L_p$ Minkowski inequality
\textbf{\cite{Lutwak96}}: If $K, L \in \mathcal{S}^n$, then
\begin{equation}\label{dminkin}
\widetilde{V}_{-p}(K,L)^n \geq V(K)^{n+p}V(L)^{-p},
\end{equation}
with equality if and only if $K$ and $L$ are dilates.

The dual $L_p$ Brunn--Minkowski inequality \textbf{\cite{Lutwak96}} is: If
$K, L \in \mathcal{S}^n$, then
\begin{equation}\label{dbminkin}
V(K \, \widetilde{+}_p \, L)^{-p/n} \geq V(K)^{-p/n}+V(L)^{-p/n},
\end{equation}
with equality if and only if $K$ and $L$ are dilates.

\section{Nonsymmetric $L_p$ Projection and Moment Bodies}

In this section we define nonsymmetric $L_p$
projection bodies $\Pi_p^+K$ as well as nonsymmetric $L_p$ moment bodies
$\mathrm{M}_p^+K$ and discuss basic properties of the corresponding operators.

Recall that the volume of the Euclidean unit ball $B$ is given by
\[\kappa_n=\pi^{n/2}/\Gamma(1+\mbox{$\frac{n}{2}$}).  \]
We define $c_{n,p}$ by
\[c_{n,p}=\frac{\Gamma\left(\mbox{$\frac{n+p}{2}$}\right)}{\pi^{(n-1)/2}\Gamma\left(\mbox{$\frac{1+p}{2}$}\right)}.  \]
For each finite Borel measure $\mu$ on $S^{n-1}$, we define a
continuous function $\mathrm{C}_p^+\mu$ on $S^{n-1}$, the {\em
nonsymmetric $L_p$ cosine transform} of $\mu$, by
\[(\mathrm{C}_p^+\mu)(u) = c_{n,p}\int_{S^{n-1}} (u \cdot v)_+^p\, d\mu(v), \qquad u \in S^{n-1},  \]
where $(u\cdot v)_+=\max\{u\cdot v, 0\}$. For $f \in C(S^{n-1})$,
let $\mathrm{C}_p^+f$ be the nonsymmetric $L_p$ cosine transform
of the absolutely continuous measure (with respect to spherical
Lebesgue measure) with density $f$. The normalization above was
chosen so that $\mathrm{C}_p^+1=1$.

The {\it nonsymmetric $L_p$ projection body} $\Pi_p^+K$ of $K \in
\mathcal{K}^n_{\mathrm{o}}$, first considered in \textbf{\cite{Lutwak96}},
is the convex body defined by
\[h(\Pi_p^+K,\cdot)^p=\mathrm{C}_p^+S_p(K,\cdot).  \]
For a star body $L \in S^{n}$, define the {\it nonsymmetric $L_p$
moment body} of $L$ by
\[h(\mathrm{M}_p^+L,\cdot)^p=\mathrm{C}_p^+\rho(L,\cdot)^{n+p}.  \]
Using polar coordinates, it is easy to verify that for $L \in
\mathcal{S}^n$,
\begin{equation}\label{mpnormal}
h(\mathrm{M}_p^+L,u)^p=c_{n,p}(n+p)\int_L (u \cdot x )_+^p\,dx, \qquad u \in S^{n-1}.
\end{equation}
Note that the normalizations are chosen such that
$\mathrm{M}_p^+B=B$ and $\Pi_p^+B=B$. For $K \in
\mathcal{K}^n_{\mathrm{o}}$, we also define
\[\mathrm{M}_p^-K = \mathrm{M}_p^+(-K) \qquad \mbox{and} \qquad \Pi_p^-K=\Pi_p^+(-K).  \]

For a finite measure $\mu$ on $S^{n-1}$, it is not hard to show
that
\[\lim \limits_{p \rightarrow 1^+}(\mathrm{C}_p^+\mu)(u) = \frac{1}{2\kappa_{n-1}} \left \{ \int_{S^{n-1}}|u\cdot v |d\mu(v) + \int_{S^{n-1}}u\cdot v d\mu(v) \right \},      \]
where the first integral is the spherical cosine transform
$\mathrm{C}\mu$ of $\mu$. Recall that pointwise convergence of
support functions on $S^{n-1}$ implies convergence in the
Hausdorff metric of the respective bodies (cf.\ \textbf{\cite[\textnormal{p.\
54}]{Schneider:CB}}). Thus, since $h(\Pi
K,\cdot)=\frac{1}{2}\mathrm{C}S(K,\cdot)$ and since area measures
have their center of mass at the origin, we obtain for every $K
\in \mathcal{K}^n_{\mathrm{o}}$ as $p \rightarrow 1$,
\begin{equation} \label{extension}
\Pi_p^+K\, \rightarrow\, \kappa_{n-1}^{-1}\Pi K \qquad \mbox{and}
\qquad \mathrm{M}_p^+K \, \rightarrow \,
\frac{n+1}{2\kappa_{n-1}}\left ( \mathrm{M}(K)+m(K)\right ).
\end{equation}
Here, $m(K)$ is up to volume normalization the centroid of $K$:
\[m(K)=\int_K x\,dx.  \]
From representation (\ref{mpnormal}), we obtain for $K \in
\mathcal{K}^n_{\mathrm{o}}$ as $p \rightarrow \infty$,
\[\mathrm{M}_p^+K\, \rightarrow\, K.   \]

A map $\Phi$ defined on $\mathcal{K}^n$ and taking values in
$\mathcal{K}^n$ is called $\mathrm{SL}(n)$ {\em covariant}, if
for all $K \in \mathcal{K}^n$ and every $\phi \in \mathrm{SL}(n)$,
\[\Phi(\phi K)=\phi \Phi K.\]
It is said to be $\mathrm{SL}(n)$ {\em contravariant}, if for all
$K \in \mathcal{K}^n$ and every $\phi \in \mathrm{SL}(n)$,
\[\Phi(\phi K)=\phi^{-\mathrm{T}} \Phi K,\]
where $\phi^{-\mathrm{T}}$ denotes the inverse of the transpose
of $\phi$.

As usual, $\Phi$ is called {\em homogeneous of degree $r$}, for
$r \in \mathbb{R}$, if $\Phi (\lambda K) = \lambda^r\Phi (K)$ for
all $K \in \mathcal{K}^n$ and every $\lambda > 0$. We say $\Phi$
{\it is linearly associating} if $\Phi$ is $\mathrm{SL}(n)$ co-
or contravariant and homogeneous of degree $r$ for some $r \in
\mathbb{R}$.

It was shown in \textbf{\cite{Ludwig:Minkowski}} that $\Pi_p^{\pm}$ is an
$n/p-1$ homogeneous and $\mathrm{SL}(n)$ contra\-variant map,
while $\mathrm{M}_p^{\pm}$ is $\mathrm{SL}(n)$ covariant and
homogeneous of degree $n/p+1$, i.e., for every $\phi \in
\mathrm{SL}(n)$ and every $\lambda > 0$,
\[\begin{array}{rcl} \Pi_p^{\pm}(\phi K)=\phi^{-\mathrm{T}} \Pi_p^{\pm} K & \,\,\mbox{ and }\,\, & \Pi_p^{\pm}(\lambda K)=\lambda^{n/p-1}\Pi_p^{\pm} K \end{array}    \]
for every $K \in \mathcal{K}^n_{\mathrm{o}}$ and
\[\begin{array}{rcl} \phantom{^{-\mathrm{T}}}\, \mathrm{M}_p^{\pm}(\phi K)=\phi \mathrm{M}_p^{\pm} K
& \,\,\mbox{ and }\,\, & \mathrm{M}_p^{\pm}(\lambda
K)=\lambda^{n/p+1}\mathrm{M}_p^{\pm} K. \end{array}  \]

A map $\Phi: \mathcal{K}^n_{\mathrm{o}} \rightarrow
\mathcal{K}^n_{\mathrm{o}}$ is called an {\it $L_p$ Minkowski
valuation} if
\[\Phi(K \cup L) +_p \Phi(K \cap L) = \Phi K +_p \Phi L,  \]
whenever $K, L, K \cup L \in \mathcal{K}^n_{\mathrm{o}}$. The
{\em trivial} $L_p$ Minkowski valuations are $L_p$ Minkowski
combinations of the identity and central reflection. In
\textbf{\cite{Ludwig:Minkowski}} Ludwig has shown that $L_p$ combinations
of $\Pi_p^{\pm}$ and $\mathrm{M}_p^{\pm}$ are the (essentially)
uniquely determined linearly associating $L_p$ Minkowski
valuations. In order to state her result, let $\mathcal{P}^n_{co}$
($\mathcal{P}^n_{\mathrm{o}}$) denote the set of polytopes in
$\mathbb{R}^n$ which contain the origin (in their interior). For
$n \geq 3$, Ludwig \textbf{\cite{Ludwig:Minkowski}} proved the following:

\begin{theorem} \label{prop} If $\Phi: \mathcal{P}^n_{co} \rightarrow
\mathcal{K}_{\mathrm{o}}^n$ is a non-trivial $L_p$ Minkowski
valuation which is linearly associating, then there exist
constants $c_0 \in \mathbb{R}$ and $c_1, c_2 \geq 0$ such that
for every $K \in \mathcal{P}^n_{\mathrm{o}}$,
\[\Phi K = \left \{ \begin{array}{ll} c_1\Pi K & \mbox{if } p = 1 \\
c_1\cdot \Pi_p^+K +_p c_2 \cdot \Pi_p^-K & \mbox{if } p
> 1 \end{array} \right.  \]
 or
\[\Phi K=\left \{ \begin{array}{ll} c_0m(K) + c_1\mathrm{M}K & \mbox{if } p = 1 \\ c_1\cdot\mathrm{M}_p^+K+_pc_2\cdot\mathrm{M}_p^-K &
\mbox{if } p > 1. \end{array} \right.\]
\end{theorem}

Theorem \ref{prop} and (\ref{extension}) show that the $L_p$
combinations of $\Pi_p^{\pm}$ and $\mathrm{M}_p^{\pm}$ are all
$L_p$ extensions of projection and moment bodies.

Ludwig's classification results in \textbf{\cite{Ludwig:Minkowski}} were
in fact formulated with a different parametrization of the
families $c_1\cdot\Pi_p^++_pc_2\cdot\Pi_p^-$ and
$c_1\cdot\mathrm{M}_p^++_pc_2\cdot\mathrm{M}_p^-$. These
alternative representations will be very useful for us as well:
For $\tau \in [-1,1]$, define the function $\varphi_{\tau}:
\mathbb{R} \rightarrow [0,\infty)$ by
\[\varphi_{\tau}(t)=|t|+\tau t,\]
and, for $K \in \mathcal{K}_{\mathrm{o}}^n$, let $\Pi_p^{\tau}K
\in \mathcal{K}_{\mathrm{o}}^n$ be the convex body with support
function
\begin{equation} \label{piptau}
h(\Pi_p^{\tau}K,u)^p=c_{n,p}(\tau)\int_{S^{n-1}}\varphi_{\tau}(u\cdot
v)^p\,dS_p(K,v), \qquad u \in S^{n-1},
\end{equation}
where
\[c_{n,p}(\tau)=\frac{c_{n,p}}{(1+\tau)^p+(1-\tau)^p}.  \]
The normalization is again chosen such that $\Pi_p^{\tau}B=B$ for
every $\tau \in [-1,1]$. From the definition of $\Pi_p^{\pm}$ it
is easy to verify that
\begin{equation} \label{taupm}
\Pi_p^{\tau}K=\frac{(1+\tau)^p}{(1+\tau)^p+(1-\tau)^p}\cdot
\Pi_p^+K +_p \frac{(1-\tau)^p}{(1+\tau)^p+(1-\tau)^p} \cdot
\Pi_p^-K.
\end{equation}
In particular, if $K \in \mathcal{K}^n_{\mathrm{o}}$ is
origin-symmetric, then for any $\tau, \sigma \in [-1,1]$, we have
$\Pi_p^{\tau}K=\Pi_p^{\sigma}K.$

By (\ref{taupm}), the one-parameter family $\Pi_p^{\tau}$
constitutes a bridge between the $L_p$ projection body operator
$\Pi_p$ (the case $\tau=0$) as introduced by Lutwak, Yang, and
Zhang and their non-symmetric analogues $\Pi_p^{\pm}$ ($\tau=\pm
1)$). From (\ref{taupm}), it also follows that for every pair
$c_1, c_2 \geq 0$ (not both zero) there exist a $\tau \in [-1,1]$
and a constant $c > 0$ such that
\begin{equation} \label{piconst}
c_1\cdot\Pi_p^+K+_pc_2\cdot\Pi_p^-K=c\,\Pi_p^{\tau} K.
\end{equation}
Thus, instead of working with the $L_p$ combinations of the
operators $\Pi_p^{\pm}$ we can consider multiples of the
operators $\Pi_p^{\tau}$, $\tau \in [-1,1]$.

For a star body $L \in \mathcal{S}^n$, let $\mathrm{M}_p^{\tau}L
\in \mathcal{K}^n_{\mathrm{o}}$ be the convex body defined by
\begin{equation} \label{defmptau}
h(\mathrm{M}_p^{\tau}
L,u)^p=c_{n,p}(\tau)\int_{S^{n-1}}\varphi_{\tau}(u\cdot
v)^p\rho(L,v)^{n+p}\,dv, \qquad u \in S^{n-1}.
\end{equation}
Then $\mathrm{M}_p^{\tau}B=B$ for every $\tau \in [-1,1]$ and
\begin{equation} \label{mtaupm}
\mathrm{M}_p^{\tau}L=\frac{(1+\tau)^p}{(1+\tau)^p+(1-\tau)^p}\cdot
\mathrm{M}_p^+L +_p \frac{(1-\tau)^p}{(1+\tau)^p+(1-\tau)^p} \cdot
\mathrm{M}_p^-L.
\end{equation}
In particular, if $L \in \mathcal{S}^n$ is origin-symmetric, then
for any $\tau, \sigma \in [-1,1]$, we have
$\mathrm{M}_p^{\tau}L=\mathrm{M}_p^{\sigma}L.$

The family $\mathrm{M}_p^{\tau}$ forms a link between $L_p$ moment
bodies (the case $\tau=0$) as introduced by Lutwak and Zhang and
their non-symmetric analogues ($\tau=\pm 1$). From
(\ref{mtaupm}), it follows that for every pair $c_1, c_2 \geq 0$
(not both zero) there exists a $\tau \in [-1,1]$ and a constant $c
> 0$ such that
\begin{equation} \label{mconst}
c_1\cdot\mathrm{M}_p^+K+_pc_2\cdot\mathrm{M}_p^-K=c\,\mathrm{M}_p^{\tau}
K.
\end{equation}
Thus, instead considering $L_p$ combinations of
$\mathrm{M}_p^{\pm}$ we can work with multiples of
$\mathrm{M}_p^{\tau}$, $\tau \in [-1,1]$.

The following simple lemma will be crucial. Here and in the
following, $\Pi_p^{\tau,*} K$ denotes the polar body of
$\Pi_p^{\tau} K$.

\begin{lemma} \label{durch} If $K  \in \mathcal{K}^n_{\mathrm{o}}$ and $L \in \mathcal{S}^n$, then
\[V_p(K,\mathrm{M}_p^{\tau}L)=\widetilde{V}_{-p}(L,\Pi_p^{\tau,*}K).  \]
\end{lemma}
\begin{proof} If $K \in \mathcal{K}^n_{\mathrm{o}}$ and $L \in
\mathcal{S}^n$, then, by (\ref{defvp}) and definition
(\ref{defmptau}),
\[V_p(K,\mathrm{M}_p^{\tau}L)=\frac{c_{n,p}(\tau)}{n}\int_{S^{n-1}}\int_{S^{n-1}}\varphi_{\tau}(u\cdot v)\rho(L,v)^{n+p}\,dv\,dS_p(K,u).\]
Thus, by Fubini's theorem, (\ref{suprad}) and definition
(\ref{piptau}),
\[V_p(K,\mathrm{M}_p^{\tau}L)=\frac{1}{n}\int_{S^{n-1}} \rho(L,v)^{n+p} \rho(\Pi_p^{\tau,*}K,v)^{-p} \,dv=\widetilde{V}_{-p}(L,\Pi_p^{\tau,*}K).\]
\end{proof}

In the following we discuss injectivity properties of the
operators $\Pi_p^+$ and $\mathrm{M}_p^+$. To this end, we first
collect some basic facts about spherical harmonics (see e.g.,
Schneider \textbf{\cite[\textnormal{Appendix}]{Schneider:CB}}).
We use $\mathcal{H}^n_k$ to denote the finite dimensional vector
space of spherical harmonics of dimension $n$ and order $k$. Let
$N(n,k)$ denote the dimension of $\mathcal{H}^n_k$.

Let $L_2(S^{n-1})$ denote the Hilbert space of square integrable
functions on $S^{n-1}$ with its usual inner product
$(\cdot\,,\cdot)$. The spaces $\mathcal{H}^n_k$ are pairwise
orthogonal with respect to this inner product. In each space
$\mathcal{H}_k^n$ we choose an orthonormal basis $\{Y_{k1},
\ldots, Y_{kN(n,k)}\}$. Then $\{Y_{k1}, \ldots, Y_{kN(n,k)}: k
\in \mathbb{N}\}$ forms a complete orthogonal system in
$L_2(S^{n-1})$, i.e., for every $f \in L_2(S^{n-1})$, the Fourier
series
\[f \sim \sum \limits_{k=0}^{\infty}\pi_k f  \]
converges in quadratic mean to $f$, where $\pi_k f$ is the
orthogonal projection of $f$ onto $\mathcal{H}_k^n$:
\[\pi_k f=\sum\limits_{i=1}^{N(n,k)} \left ( f,Y_{ki} \right ) Y_{ki}.   \]
In particular, for $f\in C(S^n)$,
\begin{equation} \label{funcdet}
\pi_kf=0 \quad \mbox{for all } k \in \mathbb{N} \qquad
\Longrightarrow \qquad f = 0.
\end{equation}
Thus, $f \in C(S^{n-1})$ is uniquely determined by its series
expansion.

For a finite Borel measure $\mu$ on $S^{n-1}$, we define
\[\pi_k \mu=\sum\limits_{i=1}^{N(n,k)} \int_{S^{n-1}}Y_{ki}(u)\,d\mu(u)\, Y_{ki}.  \]
If $f \in C(S^{n-1})$, then
\[\left ( f,\pi_k \mu \right ) = \int_{S^{n-1}} (\pi_k f)(u)\,d\mu(u).   \]
Thus, by  (\ref{funcdet}), the measure $\mu$ is uniquely
determined by its (formal) series expansion:
\begin{equation} \label{mudet}
\pi_k \mu=0 \quad \mbox{for all } k \in \mathbb{N} \qquad
\Longrightarrow \qquad \mu = 0.
\end{equation}

Of particular importance for us is the Funk--Hecke theorem: Let
$\phi$ be a continuous function on $[-1,1]$. If
$\mathrm{T}_{\phi}$ is the transformation on the set of finite
Borel measures on $S^{n-1}$ defined by
\[(\mathrm{T}_{\phi}\mu)(u)=\int_{S^{n-1}}\phi(u\cdot v)\,d\mu(v),  \]
then there are real numbers $a_k[\mathrm{T}_{\phi}]$, the {\em
multipliers} of $\mathrm{T}_{\phi}$, such that
\[\mathrm{T}_{\phi}Y_k = a_k[\mathrm{T}_{\phi}]\,Y_k \]
for every spherical harmonic $Y_k \in \mathcal{H}_k^n$. In
particular, by Fubini's theorem,
\begin{equation} \label{mult}
\pi_k \left ( \mathrm{T}_{\phi}\mu \right
)=a_k[\mathrm{T}_{\phi}]\pi_k \mu.
\end{equation}

We call a transformation $\mathrm{T}$ defined on the space of
finite Borel measures on $S^{n-1}$ and satisfying (\ref{mult}) a
{\em multiplier transformation}. Using (\ref{mudet}) and
(\ref{mult}), it follows that a multiplier transformation
$\mathrm{T}_{\phi}$ is injective if and only if all multipliers
$a_k[\mathrm{T}_{\phi}]$ are non-zero.

By the Funk--Hecke theorem, the nonsymmetric $L_p$ cosine
transform $\mathrm{C}_p^+$ is a multiplier transformation. The
numbers $a_k[\mathrm{C}_p^+]$ have been calculated in
\textbf{\cite{Rubin2000}}, see also \nolinebreak \textbf{\cite{haberl08}}: If $p$
is not an integer, then
\begin{equation} \label{pnotint}
a_k[\mathrm{C}_p^+] \neq 0,
\end{equation}
and if $p \in \mathbb{N}$, then
\begin{equation} \label{pint}
a_k[\mathrm{C}_p^+]=0 \mbox{ if and only if } k=2+p, 4+p, 6+p,
\ldots
\end{equation}
Consequently, $\mathrm{C}_p^+$ is injective if and only if $p$ is
not an integer.

Since a convex body $K \in \mathcal{K}_{\mathrm{o}}^n$ is
uniquely determined by its support function, by its radial
function and by its $L_p$ surface area measure, we conclude that
the operators $\Pi_p^+$ and $\mathrm{M}_p^+$ are injective if and
only if $p$ is not an integer. It is easy to verify that
\[\Pi_p^-K=\Pi_p^+(-K)=-\Pi_p^+K \qquad \mbox{and} \qquad \mathrm{M}_p^-K = \mathrm{M}_p^+(-K)=-\mathrm{M}_p^+K.  \]
Thus, the injectivity properties of $\Pi_p^+$ and $\mathrm{M}_p^+$
carry over to $\Pi_p^-$ and $\mathrm{M}_p^-$.

We will frequently use the following consequence of
(\ref{pnotint}) and (\ref{pint}).

\begin{lemma} \label{inject} If $K \in \mathcal{K}^n_{\mathrm{o}}$, $L \in \mathcal{S}^n$ and $p$ is not
an odd integer, then
\[\Pi_p^+K=\Pi_p^-K \qquad \mbox{or} \qquad \mathrm{M}_p^+L=\mathrm{M}_p^-L   \]
holds if and only if $K$, respectively $L$, is origin-symmetric.
\end{lemma}
\begin{proof} From the definition of $\Pi_p^-$ and $\mathrm{M}_p^-$,
it follows that $\Pi_p^+K=\Pi_p^-K$ and
$\mathrm{M}_p^+L=\mathrm{M}_p^-L$ for origin-symmetric bodies $K$
and $L$ .

Conversely, assume that $\Pi_p^+K=\Pi_p^-K$. Then $\Pi_p^+K$ is
origin-symmetric, i.e., $h(\Pi_p^+K,\cdot)^p$ is even. Note that
$f \in C(S^{n-1})$ (or a measure $\mu$ on $S^{n-1}$) is even if
and only if $\pi_k f=0$ (or $\pi_k \mu$ = 0, respectively) for
every odd $k \in \mathbb{N}$.

Since $\mathrm{C}_p^+$ is a multiplier transformation, we obtain
from (\ref{pnotint}) and (\ref{pint}) that $S_p(K,\cdot)$ is even.
Thus, by the uniqueness property of $S_p(K\cdot)$, the body $K$
must be origin-symmetric.

The case $\mathrm{M}_p^+L=\mathrm{M}_p^-L$ is similar, using
$\rho(L,\cdot)^{n+p}$ instead of $S_p(K,\cdot)$. \phantom{------}
\end{proof}

\section{Class reduction}

A standard method for establishing geometric inequalities is to
prove them first for a dense class of bodies (e.g, polytopes or
smooth bodies) and then, by taking the limit, the inequality is
obtained for all bodies. This approach has the major disadvantage
that critical equality conditions are usually lost for the limiting
case. In order to prove affine isoperimetric inequalities {\em
along} with their equality conditions for all convex bodies, it is
often sufficient to establish the inequalities only for a very
small class of bodies, e.g., the class of $L_p$ moment bodies.
This class reduction technique was introduced by Lutwak
\textbf{\cite{lutwak86}} and further applied in \textbf{\cite{LYZ2000}} and
\textbf{\cite{LZ1997}}.

The crucial result in this section, Lemma \ref{ineqgaphi}, shows
that in order to establish Theorem \ref{main1}, we need only
prove it for the class of smooth convex bodies (in fact the much
smaller class of $L_p$ moment bodies will suffice). The tools to
derive this fact are provided by Lemma \ref{durch} and the
following lemma.

\begin{lemma}\label{smooth}
If $K \in \mathcal{K}^n_{\mathrm{o}}$, then the convex body
$\mathrm{M}_p^{\tau} K$ is smooth.
\end{lemma}

\begin{proof} In order to show that $\mathrm{M}_p^{\tau} K$ is smooth, we need to prove that its support function
$h:=h(\mathrm{M}_p^{\tau}K,\cdot)$ is of class $C^2$ and that the
convex body $\mathrm{M}_p^{\tau} K$ has everywhere positive radii
of curvature (see \textbf{\cite[\textnormal{p.\ 111}]{Schneider:CB}}). To this end, we
first assume that $\tau=1$, i.e., $h=h(\mathrm{M}_p^+K,\cdot)$.
Let $f$ be a continuous function on $\mathbb{R}^n$ and let $u \in
\mathbb{R}^n\backslash \{0\}$. A simple calculation shows that
\begin{equation}\label{diff}
\frac{\partial}{\partial u_i}\int_{K}(u\cdot
x)_+^pf(x)\,dx=p\int_{K}(u\cdot x)_+^{p-1}x_i f(x)\,dx.
\end{equation}
Thus, the function $h$ is of class $C^2$ if $\tau=1$. Let
$(h_{i\!j})_{i,j=1}^{n-1}$ denote the Hessian matrix of $h$ at
$u$ with respect to an orthonormal basis $\{e_1,\ldots,e_n\}$ of
$\mathbb{R}^n$ with $e_n=u$. By \textbf{\cite[\textnormal{Corollary
2.5.3}]{Schneider:CB}}, the convex body $\mathrm{M}_p^{\tau} K$ has
everywhere positive radii of curvature if and only if
\[\det (h_{i\!j}(u))_{i,j=1}^{n-1}>0.\]
Using (\ref{diff}), we obtain for $h_{i\!j}(u)$ up to some
positive constant
\begin{eqnarray*}
\int_K(x\cdot u)_+^p\,dx\int_K(x\cdot u)_+^{p-2}(x\cdot
b_i)(x\cdot b_j)\,dx\hspace{3cm}&&\\
-\int_K(x\cdot u)_+^{p-1}(x\cdot
b_i)\,dx\int_K(x\cdot u)_+^{p-1}(x\cdot b_j)\,dx.&&
\end{eqnarray*}
An application of H\"older's inequality shows that
$(h_{i\!j})_{i,j=1}^{n-1}$ is positive definite and thus, in
particular, $\det(h_{i\!j}(u))_{i,j=1}^{n-1}>0$. Hence,
$\mathrm{M}_p^+ K$ is smooth and, since $\mathrm{M}_p^-
K=-\mathrm{M}_p^+K$, we also obtain that $\mathrm{M}_p^-K$ is
smooth. For $\tau\in(-1,1)$, the assertion follows from a similar
(but more tedious) calculation, by using (\ref{mtaupm}) and
(\ref{diff}).
\end{proof}

The crucial result of this section is contained in the following
lemma which reduces the proof of Theorem \ref{main1} to the class
of smooth convex bodies.

\begin{lemma} \label{ineqgaphi} In order to prove Theorem \ref{main1}, it is
sufficient to verify the following assertion: If $K \in
\mathcal{K}^n_{\mathrm{o}}$ is smooth, then for every
$\tau\in[-1,1]$,
\begin{equation*}
V(K)^{n/p-1}V(\Pi_p^{\tau,*} K) \leq V(B)^{n/p},
\end{equation*}
with equality if and only if $K$ is an ellipsoid centered at the
origin.
\end{lemma}
\begin{proof} For $K \in \mathcal{K}^n_{\mathrm{o}}$, let $\Phi_p
K=c_1\cdot\Pi_p^+K+_pc_2\cdot\Pi_p^-K$, where $c_1, c_2 \geq 0$
are not both zero. By (\ref{piconst}), there exist a $\tau \in
[-1,1]$ and a constant $c > 0$ such that $\Phi_p K =c\,\Pi_p^{\tau}
K.$ Since $\Pi_p^{\tau}B=B$, we conclude, that the assertion of
Theorem \ref{main1} is equivalent to the following statement: If
$K \in \mathcal{K}^n_{\mathrm{o}}$, then for every
$\tau\in[-1,1]$,
\begin{equation} \label{ipi}
V(K)^{n/p-1}V(\Pi_p^{\tau,*} K) \leq V(B)^{n/p},
\end{equation}
with equality if and only if $K$ is an ellipsoid centered at the
origin.

It remains to show that inequality (\ref{ipi}) along with its
equality conditions holds if and only if it holds for smooth
bodies. To this end, we will prove that, for $K \in
\mathcal{K}^n_{\mathrm{o}}$,
\begin{equation} \label{ppineq}
V(K)^{n/p-1}V(\Pi_p^{\tau,*} K) \leq
V(\mathrm{M}_p^{\tau}\Pi_p^{\tau,*}K)^{n/p-1}V(\Pi_p^{\tau,*}
\mathrm{M}_p^{\tau}\Pi_p^{\tau,*}K),
\end{equation}
with equality if and only if $K$ and
$\mathrm{M}_p^{\tau}\Pi_p^{\tau,*}K$ are dilates. Thus, by Lemma
\ref{smooth}, any convex body at which $
V(K)^{n/p-1}V(\Pi_p^{\tau,*} K)$ attains a maximum must be smooth.

In order to see (\ref{ppineq}), take $L=\Pi_p^{\tau,*} K$ in
Lemma \ref{durch} and use (\ref{vol2}) to conclude
\begin{equation*}
V(\Pi_p^{\tau,*} K)= V_{p}(K,\mathrm{M}_p^{\tau} \Pi_p^{\tau,*}K).
\end{equation*}
Thus, by the $L_p$ Minkowski inequality (\ref{minkin}), we obtain
\begin{equation}\label{i2}
V(\Pi_p^{\tau,*} K)^n \geq V(K)^{n-p} V(\mathrm{M}_p^{\tau}
\Pi_p^{\tau,*} K)^{p},
\end{equation}
with equality if and only if $K$ and
$\mathrm{M}_p^{\tau}\Pi_p^{\tau,*} K$ are dilates. Conversely,
replace $K$ by $\mathrm{M}_p^{\tau} L$, for some star body $L$,
in Lemma \ref{durch} and use (\ref{vol1}) to obtain
\begin{equation*}
V(\mathrm{M}_p^{\tau} L)= \widetilde{V}_{-p}(L,\Pi_p^{\tau,*}
\mathrm{M}_p^{\tau} L).
\end{equation*}
Thus, the dual $L_p$ Minkowski inequality (\ref{dminkin}) yields
\begin{equation}\label{i0}
V(\mathrm{M}_p^{\tau}L)^n \geq V(L)^{n+p}V(\Pi_p^{\tau,*}
\mathrm{M}_p^{\tau} L)^{-p},
\end{equation}
with equality if and only if $L$ and
$\Pi_p^{\tau,*}\mathrm{M}_p^{\tau} L$ are dilates.

Now take $L=\Pi_p^{\tau,*} K$ in (\ref{i0}) to get
\begin{equation}\label{i1}
V(\mathrm{M}_p^{\tau}\Pi_p^{\tau,*}K)^{n} \geq
V(\Pi_p^{\tau,*}K)^{n+p} V(\Pi_p^{\tau,*}
\mathrm{M}_p^{\tau}\Pi_p^{\tau,*} K)^{-p},
\end{equation}
with equality if and only if $\Pi_p^{\tau,*} K$ and
$\Pi_p^{\tau,*} \mathrm{M}_p^{\tau} \Pi_p^{\tau,*} K$ are dilates.

A combination of inequalities (\ref{i2}) and (\ref{i1}) finally
yields (\ref{ppineq}) and finishes the proof.
\end{proof}

By (\ref{extension}), the case $p = 1$ of inequality (\ref{ipi})
reduces to the classical Petty projection inequality. Since we do
not wish to reprove this classical inequality, we note again that
we restrict our attention to the case $1 < p < \infty$.

In Section 6, we will again use the class reduction technique to
show that Theorem \ref{main2} follows from Theorem \ref{main1}.

\section{Steiner Symmetrization and $\Pi_p^{\tau,*}$}

In this section we establish the important fact that Steiner
symmetrization intertwines with the operator $\Pi_p^{\tau,*}$ for
every $\tau \in [-1,1]$. This was proved in \textbf{\cite{LYZ2000}} for
the case $\tau=0$. For arbitrary $\tau \in [-1,1]$, the proof is
similar but \nolinebreak certain modifications are needed to
settle the equality conditions in Theorem \ref{main1}.

In the following let $\{e_1, \ldots ,e_n\}$ be an orthonormal
basis of $\mathbb{R}^n$. We will frequently use the decomposition
$\mathbb{R}^n = \mathbb{R}^{n-1} \times \mathbb{R}$, where we
assume that $e_n^{\bot} = \mathbb{R}^{n-1}$. Clearly, for every
convex body $K \in \mathcal{K}^n_{\mathrm{o}}$ there exist
functions $\underline{z}, \overline{z}: K|e_n^{\bot} \rightarrow
\mathbb{R}$ such that $K$ can be represented in the form
\begin{equation} \label{rep7}
K=\{(x,t) \in \mathbb{R}^{n-1} \times \mathbb{R}:
\underline{z}(x) \leq t \leq \overline{z}(x), x \in
K|e_n^{\bot}\}.
\end{equation}
Note that the number $\overline{z}-\underline{z}$ is the length
of the chord of $K$ through $x$ parallel to $e_n$. It is easy to
verify that $\underline{z}$ is convex and that $\overline{z}$ is
a concave function. Thus, $\underline{z}$ and $\overline{z}$ are
continuous on $K_{\mathrm{o}} := \mbox{relint}\, K|e_n^{\bot}$.
If $K$ is smooth, then $\underline{z}$ and $\overline{z}$ are
$C^1$ functions on $K_{\mathrm{o}}$.

Let $D \subseteq \mathbb{R}^{n-1}$ be an open convex set which
contains the origin in its interior. For a $C^1$ function $z: D
\rightarrow \mathbb{R}$ define
\[\langle z \rangle(x)=z(x)-x\cdot\nabla z(x), \qquad x \in D.\]
Note that the operator $\langle \cdot \rangle$ is linear.
Moreover, the kernel of $\langle \cdot \rangle$ consists only of
linear functions:
\begin{equation}\label{kernel}
\langle z \rangle(x)=0 \,\, \mbox{ for all } x \in D \qquad
\Longrightarrow \qquad z \mbox{ is linear on } D.
\end{equation}

The following auxiliary result can be found in \textbf{\cite[}\textnormal{Lemma
11}]{LYZ2000}}.

\begin{lemma} \label{supfrep} If $K \in \mathcal{K}^n_{\mathrm{o}}$ is a smooth
convex body given by
\begin{equation*}
K = \{(x,t) \in \mathbb{R}^{n-1} \times \mathbb{R}:
\underline{z}(x)\leq t\leq \overline{z}(x),\,\, x \in
K|e_n^{\bot}\},
\end{equation*}
then for every $x \in \mathrm{relint}\, K|e_n^{\bot}$,
\[h(K,(\nabla \underline{z}(x),-1))=\langle -\underline{z}
\rangle(x) \qquad \mbox{and} \qquad h(K,(-\nabla
\overline{z}(x),1))=\langle \overline{z} \rangle(x).\]
\end{lemma}

Recall that for smooth $K \in \mathcal{K}^n_{\mathrm{o}}$, the
surface area measure $S(K,\cdot)$ and thus, by (\ref{absc}), also
the $L_p$ surface area measure $S_p(K,\cdot)$ are absolutely
continuous with respect to spherical Lebesgue measure:
\begin{equation} \label{absc2}
dS_p(K,u)=h(K,u)^{1-p}f(K,u)\,du, \qquad u \in S^{n-1}.
\end{equation}
Here $f(K,\cdot)$ is the curvature function of the smooth convex
body $K$.

For smooth $K \in \mathcal{K}^n_{\mathrm{o}}$, the {\it spherical
image map} $\nu: \mathrm{bd}\,K \rightarrow S^{n-1}$ is defined
by letting $\nu(x)$, for $x \in \mathrm{bd}\,K$, be the unique
outer unit normal vector of $K$ at $x$. By \textbf{\cite[\textnormal{p.\
112}]{Schneider:CB}}, for any integrable function $g$ on $S^{n-1}$
we have
\[\int_{S^{n-1}} g(u)f(K,u)\,du = \int_{\mathrm{bd}\,K} g(\nu(x))\,d\mathcal{H}^{n-1}(x),  \]
where $\mathcal{H}^{n-1}$ denotes $(n-1)$-dimensional Hausdorff
measure. Thus, by (\ref{piptau}) and (\ref{absc2}), we obtain the
following representation of $\Pi_p^{\tau} K$, $\tau \in [-1,1]$:
\begin{equation} \label{rep6}
h(\Pi_p^{\tau} K,u)^p=c_{n,p}(\tau)\int_{\mathrm{bd}\,
K}\varphi_{\tau}(u \cdot
\nu(x))^ph(K,\nu(x))^{1-p}\,d\mathcal{H}^{n-1}(x).
\end{equation}
If the smooth convex body $K$ is given by (\ref{rep7}), then for
any continuous function $h$ on $S^{n-1}$,
\begin{eqnarray*}
\int_{\mathrm{bd} K} h(\nu(x))\,d\mathcal{H}^{n-1}(x)\hspace{8cm}&&\\
=\int_{K_{\mathrm{o}}} h(\nu(x,\underline{z}(x)))\sqrt{1+\|\nabla \underline{z}(x)\|^2}
+  h(\nu(x,\overline{z}(x)))\sqrt{1+\|\nabla
\overline{z}(x)\|^2}\,dx.&&
\end{eqnarray*}
Recall that
$K_\mathrm{o}=\mathrm{relint}\, K|e_n^{\bot}$. Since for any $x
\in K_\mathrm{o}$,
\[\nu(x,\underline{z}(x))=\frac{(\nabla\underline{z}(x),-1)}{\sqrt{1+\|\nabla\underline{z}(x)\|^2}}
\qquad \mbox{and} \qquad
\nu(x,\overline{z}(x))=\frac{(-\nabla\overline{z}(x),1)}{\sqrt{1+\|\nabla\overline{z}(x)\|^2}},\]
we obtain from Lemma \ref{supfrep}, (\ref{rep6}), and the
homogeneity of $h(K,\cdot)$ and $\varphi_{\tau}$:

\begin{lemma}\label{finrep} If $K \in \mathcal{K}^n_{\mathrm{o}}$ is a smooth
convex body given by
\begin{equation*}
K = \{(x,t) \in \mathbb{R}^{n-1} \times \mathbb{R}:
\underline{z}(x)\leq t\leq \overline{z}(x),\,\, x \in
K|e_n^{\bot}\},
\end{equation*}
then for every $(y,t) \in \mathbb{R}^{n-1} \times \mathbb{R}$,
\begin{eqnarray*}c_{n,p}^{-1}(\tau)h(\Pi_p^{\tau} K,(y,t))^p\hspace{8cm}\\
=\int_{K_\mathrm{o}}
\varphi_{\tau}(t-y\cdot\nabla\overline{z}(x))^p \langle
\overline{z} \rangle(x)^{1-p} + \varphi_{\tau}(y \cdot
\nabla\underline{z}(x)-t)^p\langle-\underline{z}\rangle(x)^{1-p}\,dx.&&
\end{eqnarray*}
\end{lemma}

For $K \in \mathcal{K}^n_{\mathrm{o}}$ and $u \in S^{n-1}$, we
denote by $S_u K$ the Steiner symmetral of $K$ with respect to
the hyperplane $u^{\bot}$, c.f. \textbf{\cite[\textnormal{p.\ 30}]{gardner2ed}}. If $K$
is given by (\ref{rep7}), then
\[
S_{e_n}K=\{(x,t)\in \mathbb{R}^{n-1}\times
\mathbb{R}:\,\,\mbox{$\frac{1}{2}$}(\underline{z}-\overline{z})(x)\leq
t\leq \mbox{$\frac{1}{2}$}(\overline{z}-\underline{z})(x),\,\, x
\in K|e_n^{\bot}\}.
\]
Our next result forms the critical part of the proof of Theorem
\ref{main1}:

\begin{lemma} \label{st} If $K \in \mathcal{K}_{\mathrm{o}}^n$ is smooth, then for every $u \in
S^{n-1}$,
\begin{equation*}
S_u\Pi_p^{\tau,*} K \subseteq \Pi_p^{\tau,*} S_uK.
\end{equation*}
If equality holds in the above inclusion, there exists an $r\in
[0,1]$ such that the points which divide the (directed) chords of
$K$ parallel to $u$ in the proportion $r:1-r$ are coplanar.
\end{lemma}
\begin{proof} Since $\Pi_p^{\tau,*}$ is linearly associating, we can assume without loss of generality that $u=e_n$.
 Let the convex body $K$ be given by
\[ K=\{(x,t)\in \mathbb{R}^{n-1}\times
\mathbb{R}:\,\underline{z}(x)\leq t\leq \overline{z}(x),\,\, x
\in K|e_n^{\bot} \}.\] The definition of Steiner symmetrization
and (\ref{suprad}) show that
\begin{equation}\label{steineren}
S_{e_n}\Pi_p^{\tau,*} K \subseteq \Pi_p^{\tau,*} S_{e_n}K
\end{equation}
holds if and only if
\[
h(\Pi_p^{\tau} K,(y,s))=h(\Pi_p^{\tau}K,(y,t))=1 \quad \mbox{with } s \neq t
\]
implies
\[ h(\Pi_p^{\tau}S_{e_n}K,(y,\mbox{$\frac{1}{2}$}s-\mbox{$\frac{1}{2}$}t)) \leq
1.\]

Let $(y,s),(y,t)\in \mathbb{R} ^{n-1}\times \mathbb{R}$ with
$s \neq t$ and suppose that
\[h(\Pi_p^{\tau} K,(y,s))=h(\Pi_p^{\tau} K,(y,t))=1.\]
Note that the Steiner symmetral of a smooth convex body is again
smooth. Since the triangle inequality implies
\[\varphi_{\tau}(\mbox{$\frac{1}{2}$}(s - t)-y \cdot \mbox{$\frac{1}{2}$} \nabla(\overline{z} - \underline{z})(x))
\leq \frac{1}{2} \left ( \varphi_{\tau}^{\phantom{/}}( s - y \cdot
\nabla \overline{z}(x)) + \varphi_{\tau}(y \cdot \nabla
\underline{z}(x)-t) \right )\] and
\[ \varphi_{\tau}(y \cdot \mbox{$\frac{1}{2}$} \nabla (\underline{z}-\overline{z})(x)-\mbox{$\frac{1}{2}$}(s-t))
\leq \frac{1}{2}\left ( \varphi_{\tau}^{\phantom{/}} (y \cdot
\nabla \underline{z}(x)-s) + \varphi_{\tau}(t- y \cdot \nabla
\overline{z}(x)) \right),\] we obtain from Lemma \ref{finrep} and
the linearity of the operator $\langle \cdot \rangle$,
\[\renewcommand{\arraystretch}{2.0}
\begin{array}{l}
c_{n,p}^{-1}(\tau)\,h(\Pi_p^{\tau} S_{e_n} K,(y,\mbox{$\frac{1}{2}$}s-\mbox{$\frac{1}{2}$} t))^p \\
 \displaystyle \phantom{--} = \int_{K_\mathrm{o}}
\varphi_{\tau}(\mbox{$\frac{1}{2}$}(s-t)-y\cdot
\mbox{$\frac{1}{2}$} \nabla (\overline{z}-\underline{z})(x))^p
\left\langle\mbox{$\frac{1}{2}$}(\overline{z}-\underline{z})\right\rangle(x)^{1-p}\,dx\\
\phantom{----} \displaystyle +
\int_{K_{\mathrm{o}}}\varphi_{\tau}(y \cdot \mbox{$\frac{1}{2}$}
\nabla(\underline{z}-\overline{z})(x)-\mbox{$\frac{1}{2}$}(s-t))^p
\left\langle\mbox{$\frac{1}{2}$}(\overline{z}-\underline{z})\right\rangle(x)^{1-p}\,dx  \\
\displaystyle \phantom{--} \leq \frac{1}{2} \int_{K_{\mathrm{o}}}
\left ( \varphi_{\tau}^{\phantom{/}}( s - y \cdot \nabla
\overline{z}(x))
+ \varphi_{\tau}^{\phantom{/}}( y \cdot \nabla \underline{z}(x)-t) \right )^p\left\langle\overline{z}-\underline{z}\right\rangle(x)^{1-p}(x)\,dx\\
\phantom{----} \displaystyle + \frac{1}{2} \int_{K_{\mathrm{o}}}
\left ( \varphi_{\tau}^{\phantom{/}}( y \cdot \nabla
\underline{z}(x)-s) + \varphi_{\tau}^{\phantom{/}}(t- y \cdot
\nabla \overline{z}(x)) \right )^p
\left\langle\overline{z}-\underline{z}\right\rangle(x)^{1-p}\,dx.
\phantom{-----}
\end{array} \]
By the convexity of the function $t \mapsto t^p$, it follows that
for real numbers $a,b \geq 0$ and $c,d > 0$,
\begin{equation}\label{numbers}
(a+b)^p(c+d)^{1-p}\leq a^pc^{1-p}+b^pd^{1-p},
\end{equation}
with equality if and only if $ad=bc$, see \textbf{\cite[\textnormal{Lemma 8}]{LYZ2000}}.
Since $K \in \mathcal{K}_{\mathrm{o}}^n$, Lemma \ref{supfrep}
implies that
$\langle\overline{z}\rangle(x),\langle-\underline{z}\rangle(x)>0$
for every $x\in K_{\mathrm{o}}$. Thus, we obtain the desired
inequality
\[ \renewcommand{\arraystretch}{2.0}
\begin{array}{l}
c_{n,p}^{-1}(\tau)h(\Pi_p^{\tau} S_{e_n}K,(y,\mbox{$\frac{1}{2}$}s-\mbox{$\frac{1}{2}$}t))^p\\
\displaystyle  \phantom{-}\leq \frac{1}{2} \int_{K_{\mathrm{o}}} \varphi_{\tau}(s-y \cdot \nabla \overline{z}(x))^p
\langle \overline{z} \rangle(x)^{1-p}+\varphi_{\tau}(y \cdot \nabla \underline{z}(x)-s) \langle-\underline{z}\rangle(x)^{1-p}\,dx\phantom{-------}\\
\displaystyle \phantom{--} + \frac{1}{2} \int_{K_{\mathrm{o}}}\varphi_{\tau}(t - y \cdot \nabla \overline{z}(x))^p
\langle\overline{z}\rangle(x)^{1-p}+\varphi_{\tau}(y \cdot \nabla \underline{z}(x)-t)^p\langle-\underline{z}\rangle(x)^{1-p}\,dx\\
\displaystyle \phantom{-} = \frac{1}{2}
c_{n,p}^{-1}(\tau)h(\Pi_p^{\tau}K,(y,s))^p+\frac{1}{2} c_{n,p}^{-1}(\tau)h(\Pi_p^{\tau}K,(y,t))^p=c_{n,p}^{-1}(\tau).
\end{array}\]

If there is equality in (\ref{steineren}), then $h(\Pi_p^{\tau}
K,(y,s))=h(\Pi_p^{\tau} K,(y,t))=1$ with $s \neq t$ implies
$h(\Pi_p^{\tau}S_{e_n}K,(y,\mbox{$\frac{1}{2}$}s-\mbox{$\frac{1}{2}$}t))=1$.
Consequently, equality must hold in the above chain of
inequalities. The equality conditions of (\ref{numbers}) now
yield for every $x \in K_{\mathrm{o}}$,
\begin{eqnarray*}
\varphi_{\tau}( s - y \cdot \nabla \overline{z}(x))
\langle-\underline{z} \rangle(x) & = & \varphi_{\tau}(y \cdot
\nabla \underline{z}(x) - t) \langle \overline{z} \rangle(x),\\
\varphi_{\tau}(y \cdot \nabla \underline{z}(x) - s) \langle
\overline{z} \rangle(x) & = & \varphi_{\tau}(t - y \cdot \nabla
\overline{z}(x)) \langle-\underline{z}\rangle(x).
\end{eqnarray*}
Hence, for $y = 0$, we conclude that
\begin{eqnarray*}
(|s|+\tau s)\langle-\underline{z}\rangle(x)&=&(|t|-\tau t)\langle\overline{z}\rangle(x),\\
(|s|-\tau s)\langle\overline{z}\rangle(x)&=&(|t|+\tau
t)\langle-\underline{z}\rangle(x).
\end{eqnarray*}
Since $(0,s),(0,t)\in\mathrm{bd }\,\Pi_p^{\tau,*}K$, there must exist a constant $c >
0$ such that $\langle\underline{z}+c\overline{z}\rangle=0$. Thus,
by (\ref{kernel}), $\underline{z}+c\overline{z}$ has to be linear.
Define $r:=c/(c+1)$.

Since $\underline{z}+c\overline{z}$ is linear, the points which
divide the (directed) chords of $K$ parallel to $e_n$ in the
proportion $r:r-1$ are coplanar.
\end{proof}

\section{Proof of the main theorems}

We are now in a position to establish our main results. We first
complete the proof of Theorem \ref{main1}. Then we show that the
nonsymmetric $L_p$ projection bodies lead to the strongest affine
isoperimetric inequality among the family of inequalities
established in Theorem \ref{main1}. The corresponding result for
nonsymmetric $L_p$ moment bodies will be given after the proof of
Theorem \ref{main2}. We emphasize again that we are assuming
throughout that $n \geq 3$ and $1 < p < \infty$.

In order to settle the equality conditions of Theorem
\ref{main1}, we will need the following generalization of the
Bertrand--Brunn theorem due to Gruber \textbf{\cite{Gruber74}}:

\begin{theorem} \label{gruber} Let $K \in \mathcal{K}_{\mathrm{o}}^n$ be a convex body. Suppose that
for any family of parallel chords of $K$ there exists an $r\in[0,1]$ such that the points
which divide the (directed) chords of $K$ in the proportion $r:r-1$ are
coplanar. Then $K$ is an ellipsoid.
\end{theorem}

By Lemma \ref{ineqgaphi}, the following result completes the proof
of Theorem \ref{main1}:

\begin{theorem} If $K \in \mathcal{K}^n_{\mathrm{o}}$ is smooth, then for every
$\tau\in[-1,1]$,
\begin{equation*}
V(K)^{n/p-1}V(\Pi_p^{\tau,*}K) \leq V(B)^{n/p},
\end{equation*}
with equality if and only if $K$ is an ellipsoid centered at the
origin.
\end{theorem}
\begin{proof} Since Steiner symmetrization does not affect volume, we
deduce from Lemma \ref{st} that for every direction $u$,
\[V(K)^{n/p-1}V(\Pi_p^{\tau,*} K)\leq V(S_u
K)^{n/p-1}V(\Pi_p^{\tau,*} S_u K).\] If equality holds, there
exists an $r\in[0,1]$ such that the points which divide the
chords of $K$ parallel to $u$ in the proportion $r:1-r$ are
coplanar.

We can now choose a sequence of Steiner symmetrals of the convex
body $K$ which converges to $(V(K)/\kappa_n)^{1/n}B$ (see e.g.,
\textbf{\cite[\textnormal{p.\ 172}]{Gruber:CDG}}). By the
continuity and the homogeneity of $\Pi_p^{\tau,*}$, we obtain
\[V(K)^{n/p-1}V(\Pi_p^{\tau,*} K)\leq V(B)^{n/p}.\]
If equality holds, then for every direction $u$ there exists an
$r\in[0,1]$ such that the points which divide the chords of $K$
parallel to $u$ in the proportion $r:1-r$ are contained in a
subspace of codimension 1 (by the proof of Lemma \ref{st}).
\linebreak Together with Theorem \ref{gruber}, this implies that
$K$ must be an ellipsoid centered at the origin.
\end{proof}

If $K \in \mathcal{K}^n_{\mathrm{o}}$ is origin-symmetric, then
for any $\tau, \sigma \in [-1,1]$,
$\Pi_p^{\tau}K=\Pi_p^{\sigma}K$ and Theorem \ref{main1} reduces
to the $L_p$ Petty projection inequality established in
\textbf{\cite{LYZ2000}}. If $K$ is not origin-symmetric, the following
theorem shows that the
nonsymmetric operators $\Pi_p^{\pm}$ provide the strongest
inequalities:

\begin{theorem} \label{strongest} For every $K \in
\mathcal{K}^n_{\mathrm{o}}$,
\[V(\Pi_p^* K) \leq V(\Pi_p^{\tau,*} K) \leq V(\Pi_p^{\pm,*}K).\]
If $K$ is not origin-symmetric and $p$ is not an odd integer,
there is equality in the left inequality if and only if $\tau =
0$ and equality in the right inequality if and only if $\tau =
\pm 1$.
\end{theorem}
\begin{proof} We may assume that $K$ is not origin-symmetric and that $p$ is not an odd integer (otherwise
the statement is trivial or follows by approximation). Let
$-1<\tau<1$. From (\ref{suprad}) and the definition of
$\Pi_p^{\tau}$, we obtain
\begin{equation}\label{radpp}
\Pi_p^{\tau,*}K=\frac{(1+\tau)^p}{(1+\tau)^p+(1-\tau)^p}\cdot
\Pi_p^{+,*}K \, \widetilde{+}_p \,
\frac{(1-\tau)^p}{(1+\tau)^p+(1-\tau)^p} \cdot \Pi_p^{-,*}K.
\end{equation}
Here, multiplication is the dual $L_p$ scalar multiplication,
i.e., $\lambda\cdot K=\lambda^{-1/p}\,K$. Using the dual $L_p$
Brunn--Minkowski inequality (\ref{dbminkin}), we obtain
\begin{equation}\label{ineq1}
V(\Pi_p^{\tau,*}K)\leq V(\Pi_p^{\pm,*}K),
\end{equation}
with equality if and only if $\Pi_p^{+,*}K$ and $\Pi_p^{-,*}K$
are dilates which is only possible if $\Pi_p^+ K=\Pi_p^- K$. From
Lemma \ref{inject}, it follows that inequality (\ref{ineq1}) is
strict for every $\tau\in(-1,1)$ which completes the proof of the
right inequality.

In order to see the left inequality, note that the polar
coordinate formula for volume yields
\[V(\Pi_p^{\tau,*}K)=\frac{1}{n}\int_{S^{n-1}}\rho(\Pi_p^{\tau,*}K,u)^n\,du.\]
Thus, using (\ref{radpp}), we obtain
\[\frac{\partial}{\partial \tau}V(\Pi_p^{\tau,*}K)=f(\tau)\int_{S^{n-1}}\!\!\rho(\Pi_p^{\tau,*}K,u)^{n+p}\left(\rho(\Pi_p^{+,*}K,u)^{-p}-\rho(\Pi_p^{-,*}K,u)^{-p}\right)du,\]
where
\begin{equation} \label{ftau}
f(\tau)=-\frac{2(1-\tau)^{p-1}(1+\tau)^{p-1}}{((1+\tau)^p+(1-\tau)^p)^2}<0.
\end{equation}
The continuous function $\tau \mapsto V(\Pi_p^{\tau,*}K)$ must
attain a minimum on $[-1,1]$. By the first part of the proof, the
points where this minimum is attained, are contained in $(-1,1)$.
If $\bar{\tau}$ is such a point, then
\[\left . \frac{\partial}{\partial \tau}V(\Pi_p^{\tau,*}K) \right |_{\tau=\bar{\tau}}=0.  \]
By the calculation above and definition (\ref{defvminp}), this is
equivalent to
\begin{equation} \label{diff0}
\widetilde{V}_{-p}(\Pi_p^{\bar{\tau},*}K,\Pi_p^{+,*}K)=\widetilde{V}_{-p}(\Pi_p^{\bar{\tau},*}K,\Pi_p^{-,*}K).
\end{equation}
Since, for $Q, K, L \in \mathcal{S}^n$ and $\alpha, \beta > 0$,
\[\widetilde{V}_{-p}(Q,\alpha \cdot K \, \widetilde{+} \, \beta \cdot L)=\alpha \widetilde{V}_{-p}(Q,K)+\beta \widetilde{V}_{-p}(Q,L),   \]
the representation (\ref{radpp}) and the identity (\ref{diff0})
imply
\[\widetilde{V}_{-p}(\Pi_p^{\bar{\tau},*}K,\Pi_p^{\bar{\tau},*}K)=\widetilde{V}_{-p}(\Pi_p^{\bar{\tau},*}K,\Pi_p^{\bar{\tau},*}(-K)).\]
By (\ref{vol2}) and since
$\Pi_p^{\bar{\tau},*}(-K)=-\Pi_p^{\bar{\tau},*}K$, we therefore
obtain
\[V(\Pi_p^{\bar{\tau},*}K)=\widetilde{V}_{-p}(\Pi_p^{\bar{\tau},*}K,-\Pi_p^{\bar{\tau},*}K).\]
Using the dual $L_p$ Minkowski inequality (\ref{dminkin}), we
conclude that $\Pi_p^{\bar{\tau},*}K$ is origin-symmetric. By
(\ref{radpp}), this is equivalent to
\[\left((1+\bar{\tau})^p-(1-\bar{\tau})^p\right)\left(\rho(\Pi_p^{+,*}K,u)^{-p}-\rho(\Pi_p^{-,*}K,u)^{-p}\right)=0\]
for every $u\in S^{n-1}$. As before, an application of Lemma
\ref{inject}, shows that $\Pi_p^{+,*}K \neq \Pi_p^{-,*}K$. Thus,
we must have $\bar{\tau}=0$ which proves the left inequality.
\end{proof}

In view of (\ref{mconst}), our next result is a stronger version
of Theorem \ref{main2}:

\begin{theorem} \label{main2a} If $L \in \mathcal{S}^n$, then for every
$\tau\in[-1,1]$,
\begin{equation*}
V(L)^{-n/p-1}V(\mathrm{M}_p^{\tau} L) \geq V(B)^{-n/p},
\end{equation*}
with equality if and only if $L$ is an ellipsoid centered at the
origin.
\end{theorem}
\begin{proof} By definition, $\mathrm{M}_p^{\tau} L\in\mathcal{K}_{\mathrm{o}}^n$. In Theorem \ref{main1}, take $K=\mathrm{M}_p^{\tau} L$, to obtain
\begin{equation}\label{imp}
V(\Pi_p^{\tau,*} \mathrm{M}_p^{\tau} L)^{-p}\geq
V(B)^{-n}V(\mathrm{M}_p^{\tau} L)^{n-p},
\end{equation}
with equality if and only if $\mathrm{M}_p^{\tau} L$ is an
ellipsoid centered at the origin. Combine this with (\ref{i0}) and
get
\[V(L)^{-n/p-1}V(\mathrm{M}_p^{\tau} L) \geq V(B)^{-n/p}.\]
If equality holds in this inequality, then equality must hold in
(\ref{i0}) and (\ref{imp}). Consequently, $L$ and $\Pi_p^{\tau,*}
\mathrm{M}_p^{\tau} L$ are dilates and $\mathrm{M}_p^{\tau} L$ is
an ellipsoid centered at the origin. Since $\Pi_p^{\tau,*}$ is
linearly associating, this implies that $L$ is an ellipsoid
centered at the origin.
\end{proof}

Now a combination of Theorem \ref{main2} with the
Blaschke--Santal\'{o} inequality yields the Corollary stated in
the introduction.

Our final result shows that the strongest inequalities in Theorem
\ref{main2a} are provided by the nonsymmetric operators
$\mathrm{M}_p^{\pm}$:

\begin{theorem} For every $L \in \mathcal{S}^n$,
\[V(\mathrm{M}_pL) \geq V(\mathrm{M}_p^{\tau}L) \geq V(\mathrm{M}_p^{\pm}L),\]
If $L$ is not origin-symmetric and $p$ is not an odd integer,
there is equality in the left inequality if and only if $\tau = 0$
and equality in the right inequality if and only if $\tau = \pm
1$.
\end{theorem}
\begin{proof} We may again assume that $L$ is not origin-symmetric and $p$ is not an odd integer. Let $-1<\tau<1$.
Using that for $L_p$ scalar multiplication $\lambda\cdot
K=\lambda^{1/p}\,K$, an application of the $L_p$ Brunn--Minkowski
inequality (\ref{bmin}) to the representation (\ref{mtaupm})
yields
\begin{equation}\label{ineq2}
V(\mathrm{M}_p^{\tau}L) \geq V(\mathrm{M}_p^{\pm}L),
\end{equation}
with equality if and only if $\mathrm{M}_p^{+}L$ and
$\mathrm{M}_p^{-}L$ are dilates which is only possible if
$\mathrm{M}_p^{+}L=\mathrm{M}_p^{-}L$. From Lemma \ref{inject},
it follows that inequality (\ref{ineq2}) is strict for every
$\tau\in(-1,1)$ which completes the proof of the right inequality.

In order to prove the left inequality, we have to calculate the
derivative of the function $\tau \mapsto V(\mathrm{M}_p^{\tau}L)$
with respect to $\tau$: For fixed $\bar{\tau}\in(-1,1)$, note that
\[
\frac{V(\mathrm{M}_p^{\tau}L)-V_1(\mathrm{M}_p^{\tau}L,\mathrm{M}_p^{\bar{\tau}}L)}{\tau-\bar{\tau}}=
\frac{1}{n}\int_{S^{n-1}}\frac{h(\mathrm{M}_p^{\tau}L,u)-h(\mathrm{M}_p^{\bar{\tau}}L,u)}{\tau-\bar{\tau}}\,dS(\mathrm{M}_p^{\tau}L,u),\]
and
\[
\frac{V_1(\mathrm{M}_p^{\bar{\tau}}L,\mathrm{M}_p^{\tau}L)-V(\mathrm{M}_p^{\bar{\tau}}L)}{\tau-\bar{\tau}}=
\frac{1}{n}\int_{S^{n-1}}\frac{h(\mathrm{M}_p^{\tau}L,u)-h(\mathrm{M}_p^{\bar{\tau}}L,u)}{\tau-\bar{\tau}}\,dS(\mathrm{M}_p^{\bar{\tau}}L,u).
\]
From the uniform convergence of support functions and the weak
convergence of surface area measures, we deduce that the limits
\begin{equation} \label{lima}
\lim_{\tau\rightarrow\bar{\tau}}\frac{V(\mathrm{M}_p^{\tau}L)-V_1(\mathrm{M}_p^{\tau}L,\mathrm{M}_p^{\bar{\tau}}L)}{\tau-\bar{\tau}},\qquad
\lim_{\tau\rightarrow\bar{\tau}}\frac{V_1(\mathrm{M}_p^{\bar{\tau}}L,\mathrm{M}_p^{\tau}L)-V(\mathrm{M}_p^{\bar{\tau}}L)}{\tau-\bar{\tau}}
\end{equation}
exist and are both equal to
\[g(\bar{\tau}):=\frac{1}{n}\int_{S^{n-1}}\left.\frac{\partial}{\partial
\tau}h(\mathrm{M}_p^{\tau}L,u)\right|_{\bar{\tau}}\,dS(\mathrm{M}_p^{\bar{\tau}}L,u).
\]
Using the $L_p$ Minkowski inequality (\ref{minkin}) for $p=1$ in
(\ref{lima}), shows that
\[
g(\bar{\tau})\leq V(\mathrm{M}_p^{\bar{\tau}}L)^{(n-1)/n}\liminf_{\tau\rightarrow\bar{\tau}}\frac{V(\mathrm{M}_p^{\tau}L)^{1/n}-V(\mathrm{M}_p^{\bar{\tau}}L)^{1/n}}{\tau-\bar{\tau}}
\]
and
\[g(\bar{\tau})\geq V(\mathrm{M}_p^{\bar{\tau}}L)^{(n-1)/n}\limsup_{\tau\rightarrow\bar{\tau}}\frac{V(\mathrm{M}_p^{\tau}L)^{1/n}-V(\mathrm{M}_p^{\bar{\tau}}L)^{1/n}}{\tau-\bar{\tau}}.
\]
Thus, we obtain
\[g(\bar{\tau})=V(\mathrm{M}_p^{\bar{\tau}}L)^{(n-1)/n}\lim_{\tau\rightarrow\bar{\tau}}\frac{V(\mathrm{M}_p^{\tau}L)^{1/n}-V(\mathrm{M}_p^{\bar{\tau}}L)^{1/n}}{\tau-\bar{\tau}}.\]
In particular, the function $\tau\rightarrow
V(\mathrm{M}_p^{\tau}L)^{1/n}$ is differentiable at $\bar{\tau}$.
The definition of $g(\bar{\tau})$ yields
\[\frac{\partial}{\partial \tau}V(\mathrm{M}_p^{\tau}L)=\int_{S^{n-1}}\frac{\partial}{\partial
\tau}h(\mathrm{M}_p^{\tau}L,u)\,dS(\mathrm{M}_p^{\tau}L,u).\]
Using (\ref{mtaupm}), we obtain for this derivative
\[-f(\tau)\int_{S^{n-1}}h(\mathrm{M}_p^{\tau}L,u)^{1-p}\left(h(\mathrm{M}_p^{+}L,u)^{p}-h(\mathrm{M}_p^{-}L,u)^{p}\right)dS(\mathrm{M}_p^{\tau}L,u),\]
where $f(\tau)$ is given by (\ref{ftau}).

The continuous function $\tau \mapsto V(\mathrm{M}_p^{\tau}L)$
must attain a maximum on $[-1,1]$. By the first part of the proof,
the points where this maximum is attained, are contained in
$(-1,1)$. If $\bar{\tau}$ is such a point, then
\[\left . \frac{\partial}{\partial \tau}V(\mathrm{M}_p^{\tau}L) \right |_{\tau=\bar{\tau}}=0.  \]
By the calculation above and definition (\ref{defvp}), this is
equivalent to
\begin{equation} \label{diff1}
V_p(\mathrm{M}_p^{\bar{\tau}}L,\mathrm{M}_p^{+}L)=V_p(\mathrm{M}_p^{\bar{\tau}}L,\mathrm{M}_p^{-}L).
\end{equation}
Since, for $Q, K, L \in \mathcal{K}^n_{\mathrm{o}}$ and $\alpha,
\beta > 0$,
\[{V}_p(Q,\alpha \cdot K +_p \beta \cdot L)=\alpha V_p(Q,K)+\beta V_p(Q,L),   \]
the representation (\ref{mtaupm}) and the identity (\ref{diff1})
imply
\[V_{p}(\mathrm{M}_p^{\bar{\tau}}L,\mathrm{M}_p^{\bar{\tau}}L)=V_{p}(\mathrm{M}_p^{\bar{\tau}}L,\mathrm{M}_p^{\bar{\tau}}(-L)).\]
By (\ref{vol1}) and since
$\mathrm{M}_p^{\bar{\tau}}(-L)=-\mathrm{M}_p^{\bar{\tau}}L$, we
therefore obtain
\[V(\mathrm{M}_p^{\bar{\tau}}L)=V_{p}(\mathrm{M}_p^{\bar{\tau}}L,-\mathrm{M}_p^{\bar{\tau}}L).\]
Using the $L_p$ Minkowski inequality (\ref{minkin}), we conclude
that $\mathrm{M}_p^{\bar{\tau}}L$ is origin-symmetric. By
(\ref{mtaupm}), this is equivalent to
\[\left((1+\bar{\tau})^p-(1-\bar{\tau})^p\right)\left(h(\mathrm{M}_p^{+}L,u)^{p}-h(\mathrm{M}_p^{-}L,u)^{p}\right)=0\]
for every $u\in S^{n-1}$. By Lemma \ref{inject},
$\mathrm{M}_p^{+}L \neq \mathrm{M}_p^{-}L$. Thus, we must have
$\bar{\tau}=0$ which proves the left inequality.
\end{proof}

\vspace{0.5cm}

\noindent {\bf Acknowledgement.} This work was supported by the
Austrian Science Fund (FWF), within the project P\,18308,
``Valuations on convex bodies''.

%The authors want to thank Erwin
%Lutwak and Monika Ludwig for their help with the presentation of
%these results.

%\bibliography{volume}
%\bibliographystyle{amsplain}
\providecommand{\bysame}{\leavevmode\hbox to3em{\hrulefill}\thinspace}
\providecommand{\MR}{\relax\ifhmode\unskip\space\fi MR }
\providecommand{\MRhref}[2]{%
  \href{http://www.ams.org/mathscinet-getitem?mr=#1}{#2}
}
\providecommand{\href}[2]{#2}

\end{document}